\documentclass{amsart}
\def\0{\{0\}}
\let\>=\rangle
\let\<=\langle
\let\noi=\noindent
\let\sse=\subseteq

\newsymbol\varnothing 203F

\def\span{{\rm span}}
\def\Real{{\rm Re\kern1pt}}
\def\smallfrac#1#2{{\textstyle{\frac{#1}{#2}}}}

\def\H{{\mathcal H}}
\def\Le{{\mathcal L}}
\def\K{{\mathcal K}}
\def\M{{\mathcal M}}
\def\N{{\mathcal N}}
\def\R{{\mathcal R}}

\newtheorem{theorem}{Theorem}
\newtheorem{lemma}{Lemma}
\newtheorem{corollary}{Corollary}

\theoremstyle{definition}
\newtheorem{definition}{Definition}
\newtheorem{example}{Example}
\newtheorem{remark}{Remark}

\numberwithin{theorem}{section}
\numberwithin{lemma}{section}
\numberwithin{corollary}{section}
\numberwithin{proposition}{section}
\numberwithin{conjecture}{section}
\numberwithin{definition}{section}
\numberwithin{remark}{section}
\numberwithin{question}{section}
\numberwithin{example}{section}
\numberwithin{equation}{section}

\begin{document}

\vglue-55pt
\hfill{\it Bulletin of the Korean Mathematical Society}\/,
{\bf 56} (2019) 585--596

\vglue35pt
\title{Biisometric Operators and Biorthogonal Sequences}
\author{Carlos Kubrusly}
\address{Applied Mathematics Department, Federal University,
         Rio de Janeiro, Brazil}
\email{carloskubrusly@gmail.com}
\author{Nhan Levan}
\address{Electrical Engineering Department, University of California,
         Los Angeles, USA}
\email{levan@ee.ucla.edu}
\renewcommand{\keywordsname}{Keywords}
\thanks{\bf Corrections in Theorem 3.1 and Corollary 3.1 issued
            on February 11, 2021.}
\date{Submitted: March 18, 2018; accepted: April 1, 2019}

\begin{abstract}
It is shown that a pair of Hilbert space operators $V$ and $W$ such that
${V^*W=I}$ (called a biisometric pair) shares some common properties with
unilateral shifts when orthonormal basis are replaced with biorthogonal
sequences, and it is also shown how such a pair of biisometric operators
yields a pair of biorthogonal sequences which are shifted by them$.$ These
are applied to a class of Laguerre operators on $L^2[0,\infty)$.
\end{abstract}

\maketitle

\section{Introduction}

Throughout this paper $\H$ and $\K$ stand for Hilbert spaces$.$ We use the
same symbol ${\<\cdot\,;\cdot\>}$ and ${\|\cdot\|}$ for the inner product
and norm in both of them, respectively$.$ Let ${T\!:\H\to\K}$ be a bounded
linear transformation (referred to as an operator if ${\K=\H}$ --- i.e., an
operator on $\H$ is a bounded linear transformation of $\H$ into itself)$.$
Let $I$ stand for the identity operator (either on $\H$ or on $\K$)$.$ Recall:
$T$ is an isometry if ${T^*T=I}$, identity on $\H.$ Every isometry is
injective$.$ A transformation $T$ is unitary if it is a surjective isometry
(i.e., an invertible isometry), which means $T$ is an isometry and a
coisometry (i.e., ${TT^*=I}$, identity on $\K$; and ${T^*T=I}$, identity on
$\H).$ By a subspace of $\H$ we mean a {\it closed}\/ linear manifold of $\H.$
Let $\M^-$ and $\M^\perp$ denote closure and orthogonal complement,
respectively, of a linear manifold $\M$ of $\H$ (both are subspaces of $\H).$
The kernel and range of a bounded linear transformation $T$ will be denoted by
$\N(T)$ (a subspace of $\H$) and $\R(T)$ (a linear manifold of $\K$),
respectively$.$ The adjoint of $T$ (which is a bounded linear transformation
of $\K$ into $\H$) will be denoted by $T^*\!.$ Let $\span\,A$ denotes the
linear span of an arbitrary set ${A\sse\H}$ and let $\bigvee\!A$ denotes the
closure of $\span\,A$.

\vskip6pt
An operator ${S\!:\H\to\H}$ is a {\it unilateral shift}\/ if there exists an
infinite sequence $\{\H_k\}_{k=0}^\infty$ of nonzero pairwise orthogonal
subspaces of $\H$ (i.e., ${\H_j\perp\H_k}$) such that
$\H={\bigoplus}_{k=0}^\infty\H_k$ (i.e., $\{\H_k\}_{k=0}^\infty$ spans $\H$)
and $S$ maps each $\H_k$ isometrically onto $\H_{k+1}$ so that
${S|_{\H_k}\!:\H_k\to\H_{k+1}}$ is a unitary transformation (i.e., a
surjective isometry)$.$ Thus $\dim\H_{k+1}=\dim\H_k$ for every ${k\ge0}.$ Such
a common dimension is the {\it multiplicity}\/ of $S.$ The adjoint
${S^*\!:\H\to\H}$ of $S$ is referred to as a {\it backward unilateral
shift}\/$.$ Every unilateral shift $S$ is an isometry (i.e., ${S^*S=I}$) and
so $S$ is injective (but not surjective)$.$ Moreover, for each ${k\ge0}$
$$
\H_k=S^k\H_0
\quad\;\hbox{with}\;\quad
\H_0=\N(S^*),
$$
where $\N(S^*)$ denotes the kernel of $S^*\!.$ Therefore
$$
S\H_k=\H_{k+1}
\quad\;\hbox{and}\;\quad
S^*\H_{k+1}=\H_k.
$$
\vskip-2pt

\vskip6pt
In this paper we show that there exist pairs of Hilbert space operators $V$
and $W$ that satisfy the above displayed shifting properties (although they
may not be unilateral shifts themselves, not even isometries), where
orthogonality is replaced by biorthogonality$.$ The motivation behind such a
program comes from the following result from \cite{LK1,LK2}$.$ {\it If\/ $S$
and\/ $R$ are unilateral shifts on a Hilbert space $\H$ such that\/
${SS^*\!+RR^*}\!= I$, then\/ $\H$ admits the dual-shift decomposition},
namely,
$$
\H={\bigoplus}_{k=1}^\infty S^k\N(S^*)
\oplus{\bigoplus}_{k=1}^\infty R^k\N(R^*).
$$
(The symbol $\oplus$ stands for orthogonal direct sum.) This will be
approached here in light of biorthogonal sequences (which are not necessarily
individually orthogonal sequences) and biisometric operators (which are not
necessarily individually isometric operators)$.$ Next section discusses these
notions.
\vskip3pt

\section{Biorthogonal Sequences}

Biorthogonal sequences are germane to Banach spaces and were introduced in
the context of basis for separable Banach spaces \cite[Definition 1.4.1]{S},
\cite[Definition 1.f.1]{LT}$.$ Thus let $\H$ be a {\it separable}\/ Hilbert
space$.$ In a Hilbert space setting (where dual pair boils down to inner
product after the Riesz Representation Theorem) the notion of
biorthogonality reads as follows.

\vskip6pt
\begin{definition}
Two sequences $\{f_n\}$ and $\{g_n\}$ of vectors in $\H$ are said to be
{\it biorthogonal}\/ (to each other) if ${\<f_m;g_n\>}=\delta_{m,n}$ where
$\delta$ stands for the Kronecker delta function$.$ If $\{f_n\}$ is such that
there exists a sequence $\{g_n\}$ for which $\{f_n\}$ and $\{g_n\}$ are
biorthogonal, then it is said that $\{f_n\}$ {\it admits a biorthogonal
sequence}\/ (and so does $\{g_n\}$) and the pair $\{\{f_n\},\{g_n\}\}$ is
referred to as a {\it biorthonormal system}\/.
\end{definition}
\vskip-2pt

\vskip6pt
Let $\{f_n\}$ and $\{g_n\}$ be a pair of biorthogonal sequences$.$ If they are
such that $f_n=g_n$ for all $n$, then we get the definition of an orthonormal
sequence, although in general neither $\{f_n\}$ nor $\{g_n\}$ are orthogonal
(much less orthonormal) sequences.

\vskip6pt
A sequence $\{f_n\}$ that admits a biorthogonal sequence $\{g_n\}$ was
called {\it minimal}\/ in \cite[Definition 1.f.1]{LT}, where it was pointed
out that (i) $\{f_n\}$ {\it admits a biorthogonal sequence if and only if
${f_k\not\in\bigvee\{f_n\}_{n\ne k}}$ for every integer}\/ $k$ (i.e.,if and
only if each vector $f_k$ from\/ $\{f_n\}$ is not in the closure of
$\span\{f_n\}/\{f_k\}$) \cite[p.42]{LT} and (ii) {\it every basic sequence
is minimal}\/ \cite[pp.43]{LT} (a sequence is {\it basic}\/ if it is a
Schauder basis for its closed span)$.$ So {\it every orthonormal sequence
admits a biorthonormal sequence}\/.

\vskip6pt
A sequence $\{f_n\}$ spanning the whole space $\H$ is sometimes called
{\it total}\/ (or {\it com\-plete}\/, or {\it fundamental}\/)$.$ This means
$\bigvee\{f_n\}=\H$ or, equivalently, ${f\perp f_n}$ for every $n$ implies
${f=0}$ (i.e., $\{f_n\}$ is total if and only if the only vector orthogonal
to every $f_n$ is the origin)$.$ It was pointed out in \cite{Y} that (iii)
{\it if\/ $\{f_n\}$ admits a biorthogonal sequence\/ $\{g_n\}$, then\/
$\{g_n\}$ is unique if and only if\/ $\{f_n\}$ is total}\/.

\vskip6pt
If $\{f_n\}$ admits a biorthogonal sequence $\{g_n\}$, and if $\{f_n\}$ is
total, then $\{g_n\}$ is not necessarily total (i.e., the property
``$\{f_n\}$ spans $\H$'' is not inherited by $\{g_n\}$ --- as in the example
below)$.$ A total sequence $\{f_n\}$ that admits a (unique) biorthogonal
sequence $\{g_n\}$ was called {\it exact}\/ in \cite{Y}, where it was shown
that (iv) {\it if the sequence\/ $\{f_n\}=\{e^{i\,\alpha_n}\}$ of vectors in
the Hilbert space}\/ $L^2({-\pi,\pi})$ (so that $f_n(t)=e^{i\,\alpha_nt}$ for
each $n$ almost everywhere in ${(-\pi,\pi)}$ --- i.e., for almost every $t$
in ${(-\pi,\pi)}$ with respect to Lebesgue measure) {\it is exact}\/ (i.e.,
$\{f_n\}$ is total and admits a (unique) biorthogonal sequence\/ $\{g_n\}$),
{\it then\/ the biorthogonal sequence $\{g_n\}$ is also exact}\/.

\vskip6pt
\begin{example}
Let $\{e_n\}_{n\ge1}$ be any orthonormal basis for $\H$ (any orthonormal
sequences that spans $\H$, thus total).
\vskip6pt\noi
(a) The sequence $\{f_n\}_{n\ge1}=\{{e_1+e_{n+1}}\}_{n\ge1}$ is total (since
if ${f\perp f_n}$ for all $n$, then the absolute value of the Fourier
coefficients of $f$ with respect to the orthonormal basis $\{e_n\}_{n\ge1}$
are constant, and so ${f=0}$)$.$ Moreover, $\{f_n\}_{n\ge1}$ admits a
biorthogonal sequence $\{g_n\}_{n\ge1}=\{e_{n+1}\}_{n\ge1}$ which is unique
and not total.
\vskip6pt\noi
(b) The sequences $\{f_n\}_{n\ge1}=\{{e_1+e_2+e_{n+2}}\}_{n\ge1}$,
$\{g_n\}_{n\ge1}=\{e_{n+2}\}_{n\ge1}$, and
$\{h_n\}_{n\ge1}=\{{e_1-e_2+e_{n+2}}\}_{n\ge1}$ are pairwise biorthogonal to
each other, and therefore they are all not total.
\vskip6pt\noi
Furthermore, every vector from $\{f_n\}_{n\ge1}$ in (a) or in (b), and from
$\{h_n\}_{n\ge1}$ in (b), is not orthogonal to any other vector from the same
sequence, and all vectors in $\{f_n\}_{n\ge1}$ and in $\{h_n\}_{n\ge1}$ have
squared norm 2 or 3 while $\{g_n\}_{n\ge1}$ in (a) or in (b) is an
orthonormal sequence.
\end{example}
\vskip-2pt

\vskip6pt
Indeed, {\it there is no distinct pair of biorthonormal sequences}\/
as we show below.

\vskip6pt
\begin{theorem}
Take a pair of biorthogonal sequences\/ $\{f_n\}$ and\/ $\{g_n\}.$ If\/
$\|f_n\|=\|g_n\|^{-1}$, then\/ $\{f_n\}$ and\/ $\{g_n\}$ are proportional,
which means for each\/ $n$ there exists a nonnegative number\/ $\alpha_n$
for which\/ $f_n=\alpha_ng_n.$ Moreover,\/ $\alpha_n=\|f_n\|^2$.
\end{theorem}

\begin{proof}
Suppose $\{f_n\}$ and $\{g_n\}$ are biorthogonal sequence$.$ Hence
$\<{f_n,g_n}\>=1.$ Take an arbitrary $n.$ If $\|f_n\|=\|g_n\|^{-1}$, then
$$
\<f_n,g_n\>=\|f_n\|\|g_n\|.
$$
But this is equivalent to saying that (see, e.g., \cite[Problem 5.2]{EOT})
$$
f_n=\alpha_n g_n
$$
for some positive real number $\alpha_n.$ Therefore $\{f_n\}$ and $\{g_n\}$
are proportional (and so biorthogonal to each other)$.$ Moreover, since
$\|f_n\|=\|g_n\|^{-1}$, it follows that
$\|f_n\|=\alpha_n\|g_n\|=\alpha_n\|f_n\|^{-1}$ and so
${\alpha_n=\|f_n\|^2}$.
\end{proof}

\vskip6pt
\begin{corollary}
There is no distinct pair of biorthonormal sequences$.$ In other words, if two
sequences\/ $\{f_n\}$ and\/ $\{g_n\}$ are biorthogonal and if\/
$\|f_n\|=\|g_n\|=1$ for all\/ $n$, then\/ ${f_n=g_n}$ for all\/ $n$.
\end{corollary}

\begin{proof}
This is a particular case of Theorem 2.1 for $\|f_n\|=\|g_n\|=1$ for all $n$.
\end{proof}

\section{Biisometric Operators}

Consider a pair of operators $V$ and $W$ on a Hilbert space $\H$.

\vskip6pt
\begin{definition}
Two operators $V$ and $W$ are said to be {\it biisometric}\/ if ${V^*W=I}$,
and in this case $\{V,W\}$ is referred to as a {\it biisometric pair}\/ on
$\H$.
\end{definition}
\vskip-2pt

\vskip6pt
It is clear that ${V^*W=I}$ if and only if ${W^*V=I}.$ Thus $\{V,W\}$ is a
biisometric pair if and only if
$$
V^*W=I=W^*V.
$$
Let $V$ and $W$ be a pair of biisometric operators$.$ If they are such that
${W=V}\!$, then we get the definition of an isometry, viz., ${V^*V=I}$,
although in general neither $V$ nor $W$ are assumed to be isometries
themselves.

\vskip6pt
\begin{theorem}
Let\/ $V\!$ and\/ $W\!$ be operators on\/ $\H.$ Suppose their adjoints\/
$V^*\!$ and\/ $W^*\!$ are not injective$.$ Take arbitrary nonzero vectors\/
$v$ and\/ $w$ in\/ $\H.$ For each nonnega\-tive integer\/ $n$ consider the
vectors
$$
\phi_n=V^nw
\quad\;\hbox{and}\;\quad
\psi_n=W^nv
$$
in\/ $\H.$ If\/ $\{V,W\}$ is a biisometric pair on\/ $\H$, then there exist
$$
v\in\N(V^*)
\quad\;\hbox{and}\;\quad
w\in\N(W^*)
$$
such that the sequences\/ $\{\phi_n\}$ and\/ $\{\psi_n\}$ are biorthogonal
sequences\/$.$ Moreover,
$$
V\phi_n=\phi_{n+1}
\quad\;\hbox{and}\;\quad
W\psi_n=\psi_{n+1},
$$
\vskip-4pt\noi
and also
$$
V^*\psi_{n+1}=\psi_n
\quad\;\hbox{and}\;\quad
W^*\phi_{n+1}=\phi_n.
$$
\end{theorem}

\begin{proof}
Suppose $V$ and $W$ have noninjective adjoints$.$ This is equivalent to saying
that $V$ and $W$ have nondense ranges (as ${\N(T^*)^\perp\!=\R(T)^-\!}$ for
every operator $T).$ Since $V^*$ and $W^*$ are noninjective take arbitrary
nonzero vectors ${v\in\N(V^*)}$ and ${w\in\N(W^*)}$, arbitrary nonnegative
integers ${m,n}$, and set
$$
\phi_n=V^nw
\quad\,\hbox{and}\,\quad
\psi_n=W^nv.
$$
Suppose ${n<m}.$ As $V^*W\!=I$ a trivial induction leads to $V^{*n}W^n\!=I$
and so
$$
\<\phi_m\,;\psi_n\>=\<V^mw\,;W^nv\>=\<w\,;V^{*(m-n)}V^{*n}W^nv\>
=\<w\,;V^{*(m-n)}v\>=0
$$
for ${n<m}$ because ${v\in\N(V^*)}$ implies ${v\in\N(V^{*{m-n}})}.$ On the
other hand, suppose ${m<n}.$ As ${W^*V=I}$ and ${w\in\N(W^*)}$ a similar
argument ensures ${\<\phi_m\,;\psi_n\>}=0$ for ${m<n}.$ Moreover, since
$W^*V=I$ we also get the following nonorthogonality.

\vskip9pt\noi
{\it Claim}\/.\hskip108pt
$\N(V^*)\not\perp\N(W^*)$.

\vskip9pt\noi
{\it Proof}\/.
Since ${W^*V=I}$, then ${V\!\ne O}.$ Thus take any ${0\ne y=\R(V)}$ so that
${y=Vx}$ for some ${0\ne x\in\H}.$ If ${y\in\N(W^*)}$, then ${0=W^*y=W^*Vx=x}$,
which is a contra\-diction$.$ So ${\R(V)\cap\N(W^*)\kern-1pt=\kern-1pt\0}$
which implies ${\R(V)^-\kern-1pt\cap\N(W^*)\kern-1pt=\kern-1pt\0}.$ Therefore
$$
\N(V^*)^\perp\cap\N(W^*)=\0.
$$
Suppose ${\N(W^*)\kern-1pt\perp\N(V^*)}.$ Then ${\N(W^*)\sse\N(V^*)^\perp}.$
Thus by the above identity
$\N(W^*)={\N(W^*)\cap\N(W^*)}\sse{\N(V^*)^\perp\cap\N(W^*)}=\0$, which
contradicts the assumption of $W^*$ being noninjective$.$ Hence
${\N(V^*)\kern-1pt\not\perp\N(W^*)}.\!\!\qed$

\vskip6pt\noi
Thus there exist (nonzero) ${v\in\N(V^*)}$ and ${w\in\N(W^*)}$ such that
${\<w\,;v\>\ne0}$, and so we may take ${v\in\N(V^*)}$ and ${w\in\N(W^*)}$
for which ${\<\,w\,;v\>}=1.$ Then
$$
\<\phi_n\,;\psi_n\>=\<\,w\,;v\>=1.
$$
Therefore ${\<\phi_m\,;\psi_n\>}=\delta_{m,n}.$ This means $\{\phi_n\}$ and
$\{\psi_n\}$ are biorthogonal sequences$.$ Moreover, according
to definition of $\phi_n$ and $\psi_n$ we get
$$
V\phi_n=V^{n+1}w=\phi_{n+1}
\quad\;\hbox{and}\;\quad
W\psi_n=W^{n+1}v=\psi_{n+1},
$$
and since ${V^*W=I=W^*V}$ we also get
\vskip6pt\noi
$$
V^*\psi_{n+1}=V^*W^{n+1}v=W^nv=\psi_n
\quad\;\hbox{and}\;\quad
W^*\phi_{n+1}=W^*V^{n+1}w=V^nw=\phi_n,
$$
\vskip4pt\noi
for every nonnegative integer $n$.
\end{proof}
\goodbreak\noi

\vskip6pt
\begin{corollary}
Let\/ $\{V,W\}$ be a biisometric pair on\/ $\H$ and consider the biorthogonal
sequences\/ $\{\phi_n\}$ and\/ $\{\psi_n\}$ defined in Theorem 3.1 in terms
of nonzero vectors\/ ${v\in\N(V^*)}$ and\/ ${w\in\N(W^*)}.$ In addition, if
these biorthogonal sequences span\/ $\H$, then every\/ ${x\in\H}$ with
series expansion in\/ $\{\phi_n\}$ and\/ $\{\psi_n\}$ can be expressed as
$$
x={\sum}_{k=0}^{\infty}\<x\,;\psi_k\>\phi_k
={\sum}_{k=0}^{\infty}\<x\,;\phi_k\>\psi_k,
$$
\vskip-4pt\noi
and therefore
$$
Vx={\sum}_{k=0}^{\infty}\<x\,;\psi_k\>\phi_{k+1}
\quad\;\hbox{and}\;\quad
Wx={\sum}_{k=0}^{\infty}\<x\,;\phi_k\>\psi_{k+1},
$$
$$
V^*x={\sum}_{k=0}^{\infty}\<x\,;\phi_{k+1}\>\psi_k
\quad\;\hbox{and}\;\quad
W^*x={\sum}_{k=0}^{\infty}\<x\,;\psi_{k+1}\>\phi_k.
$$
\end{corollary}

\begin{proof}
Take an arbitrary ${x\in\H}.$ If the biorthogonal sequences $\{\phi_n\}$ and
$\{\psi_n\}$ span $\H$ (i.e., if $\bigvee\{\phi_n\}=\bigvee\{\psi_n\}=\H$)
and if
$$
x={\sum}_{k=0}^{\infty}\alpha_k\phi_k
={\sum}_{k=0}^{\infty}\beta_k\psi_k
$$
for some pair of sequences of scalars $\{\alpha_n\}$ and $\{\beta_n\}$, then
$$
\alpha_n=\<x\,;\psi_n\>
\quad\;\hbox{and}\;\quad
\beta_n=\<x\,;\phi_n\>
$$
for every ${n\ge0}.$ Indeed, observe (by the continuity of the inner product
and recalling that $\{\phi_k\}$ and $\{\psi_k\}$ are biorthogonal) that
$$
\<x\,;\psi_n\>={\sum}_{k=0}^{\infty}\alpha_k\<\phi_k\,;\psi_n\>=\alpha_n
\quad\;\hbox{and}\;\quad
\<x\,;\phi_n\>={\sum}_{k=0}^{\infty}\alpha_k\<\psi_k\,;\phi_n\>=\beta_n
$$
for every ${n\ge0}.$ Finally, recall from Theorem 3.1 that
for each ${n\ge0}$
$$
\hbox{${V\phi_n=\phi_{n+1}}$, $\,\;{W\psi_n=\psi_{n+1}}$,
$\,\;{V^*\psi_{n+1}=\psi_n}$, $\,\;$and $\;\;{W^*\phi_{n+1}=\phi_n}$}.
$$
So apply $V$ and $W^*$ to the expansion of $x$ in terms of $\{\phi_n\}$ and
apply $V^*$ and $W$ to the expansion of $x$ in terms of $\{\psi_n\}$ (using
the continuity of the inner product).
\end{proof}

\section{Laguerre Shifts}

We now apply the results of Section 3 to a class of Laguerre operators$.$
Recall that the Laguerre functions consist of an orthonormal basis
$\{e_n\}_{n=0}^\infty$ for the concrete Hilbert space ${L^2[0,\infty)}$
(see, e.g., \cite[Example 5.L(d)]{EOT}) defined a.e. for ${t\ge0}$ by
$$
e_n(t)=e^{-\frac{1}{2}t}\,L_n(t)
$$
for each integer ${n\ge0}$, where $L_n(t)$ are the Laguerre polynomials of
degree ${n\ge0}.$ Consider the operator ${S\!:L^2[0,\infty)\to L^2[0,\infty)}$
defined by $Sf=g$, where (for almost all ${t\ge0}$ with respect to Lebesgue
measure; i.e., almost everywhere (a.e.) on ${[0,\infty)}$)
$$
(Sf)(t)=g(t)
\quad\;\hbox{with}\;\quad
g(t)=f(t)-\int_0^t\!e^{-\frac{1}{2}(t-\tau)}\,f(\tau)\,d\tau,
$$
which is an isometry having the shift property, viz., (with
$e_n(t)=e^{-\frac{1}{2}t}\,L_n(t)\,$),
$$
(Se_n)(t)=e^{-\frac{1}{2}t}L_{n+1}(t)=e_{n+1}(t)
$$
for every ${t\geq 0}$ and each integer ${n\ge0}.$ This is referred to as
the Laguerre shift (of multiplicity $1$) generating the Laguerre functions$.$
Let ${D_{2\alpha}\!:L^2[0,\infty)\to L^2[0,\infty)}$ be the
dilation-by-$2\alpha$-operator defined by ${D_{2\alpha}f=g}$ where for every
${t\ge0}$
$$
(D_{2\alpha}f)(t)=g(t)
\quad\;\hbox{with}\;\quad
g(t)=\sqrt{2\alpha}\,f(2\alpha t)
$$
for each ${\alpha\ge\frac{1}{2}}.$ The $\alpha$-Laguerre functions are then
defined for each ${n\ge0}$ (again with $e_n(t)=e^{-\frac{1}{2}t}\,L_n(t)\,$)
by
$$
(D_{2\alpha}\,e_n)(t)
=\sqrt{2\alpha}\,e_n(2\alpha t)
=\sqrt{2\alpha}\,e^{-\alpha t}L_{n}(2\alpha t)
$$
for every ${t\ge0}.$ Similarly, the $\alpha$-Laguerre shift $S_{\alpha}$ ---
generating the $\alpha$-Laguerre functions --- is defined by
${S_{\alpha}=D_{2\alpha}S}$ so that for every ${t\ge0}$
$$
(S_{\alpha}f)(t)=(D_{2\alpha}S)f(t)=g(t)
\quad\;\hbox{with}\;\quad
g(t)=f(t)-2\alpha\int_0^t\!e^{-\alpha(t-\tau)}f(\tau)\,d\tau
$$
and (with
$(D_{2\alpha}\,e_n)(t)=\sqrt{2\alpha}\,e^{-\alpha t}L_{n}(2\alpha t)\,$)
for each ${n\ge0}$
$$
(S_{\alpha}\,e_n)(t)=(D_{2\alpha}Se_n)(t)=(D_{2\alpha}e_{n+1})(t)
=\sqrt{2\alpha}\,e^{-\alpha t}\,L_{n+1}(2\alpha t).
$$
Now consider a class of 2-parameter
Laguerre functions as follows$.$ The $(\alpha+\beta)$-Laguerre functions are
defined for every ${t\ge0}$ by (recall: $e_n(t)=e^{-\frac{1}{2}t}\,L_n(t)\,$)
$$
(D_{\alpha+\beta}\,e_n)(t)
=\sqrt{\alpha +\beta}\,e^{-\frac{\alpha+\beta}{2}t}
L_n\big((\alpha+\beta)\kern1pt t\big),
$$
for each ${n\ge0}$ where ${\alpha,\beta\ge\frac{1}{2}}.$
(Compare with \cite[Section 5]{KL}.) From now on we proceed formally.

\vskip6pt
\begin{lemma}
The sequences\/ $\{\phi_n\}$ and\/ $\{\psi_n\}$ given by
$$
\phi_n(t)
=\sqrt{\alpha+\beta}\,e^{-\alpha t}L_n\big((\alpha+\beta)\kern1pt t\big)
\quad\;\hbox{and}\;\quad
\psi_n(t)
=\sqrt{\alpha+\beta}\ e^{-\beta t}\,L_n\big((\alpha+\beta)\kern1pt t\big)
$$
for each\/ ${n\ge0}$ are biorthogonal and span\/ $L^2[0,\infty)$.
\end{lemma}

\begin{proof}
For each ${m,n\ge0}$
\begin{eqnarray*}
\<\phi_m\,;\psi_{n}\>
\kern-6pt&=&\kern-6pt
\!\!\int_0^\infty\!\!\!\!
\sqrt{\alpha\!+\!\beta}\,e^{-\alpha t}
L_m\big((\alpha+\beta)\kern1pt t\big)
\sqrt{\alpha+\beta}\,e^{-\beta t}
L_n\big((\alpha\!+\!\beta)\kern1pt t\big)\,dt \\
\kern-6pt&=&\kern-6pt
\!\!\int_0^\infty\!\!\!\!
\sqrt{\alpha\!+\!\beta}\,e^{-\frac{\alpha+\beta}{2}t}
L_m\big((\alpha\!+\!\beta)\kern1pt t\big)
\sqrt{\alpha+\beta}\,e^{-\frac{\alpha+\beta}{2}t}
L_n\big((\alpha+\beta)\kern1pt t\big)\,dt
=\delta_{m,n}
\end{eqnarray*}
since the above-defined $(\alpha+\beta)$-Laguerre functions, namely,
${D_{\alpha+\beta}\big(e^{-\frac{1}{2}t}L_n(t)\big)}
={\sqrt{\alpha+\beta}\,e^{-\frac{\alpha+\beta}{2}t}}
{L_n\big((\alpha+\beta)\kern1pt t\big)}$,
are orthonormal$.$ It remains to show that $\{\phi_n\}$ and $\{\psi_n\}$
span $L^2[0,\infty).$ Suppose there is a nonzero ${h\in L^2[0,\infty)}$
such that
$$
\int_0^\infty\!\!e^{-\alpha t}L_n\big((\alpha+\beta)\kern1pt t\big)h(t)\,dt=0
$$
for every ${t\ge0}$ and each ${n\ge0}$, which can be rewritten as
$$
\int_0^\infty\!\!e^{-\frac{\alpha+\beta}{2}t}L_n\big((\alpha+\beta)\kern1pt t
\big)\big(e^{-\frac{\alpha-\beta}{2}t}h(t)\big)\,dt=0.
$$
So, as the $(\alpha+\beta)$-Laguerre functions
${\sqrt{\alpha+\beta}\,e^{-\frac{\alpha+\beta}{2}t}
L_n\big((\alpha+\beta)\kern1pt t\big)}$
span $L^2[0,\infty)$,
$$
e^{-\frac{\alpha-\beta}{2}t}h(t)=0
\;\;\hbox{for every}\;\;
t\ge0
\quad\;\Longrightarrow\;\quad
h(t)=0
\;\;\hbox{for every}\;\;
t\ge0,
$$
ensuring that $\{\phi_n\}$ spans $L^2[0,\infty).$ Interchanging $\alpha$ and
$\beta$, $\{\psi_n\}$ spans $L^2[0,\infty)$.$\!\!\!$
\end{proof}

\vskip6pt
\begin{lemma}
The Laplace transforms of each\/ $\phi_n$ and\/ $\psi_n$ are given by
$$
\Le_s[\phi_n](s)
=\left[\smallfrac{s-\beta}{s+\alpha}\right]^n
\smallfrac{\sqrt{\alpha+\beta}}{s+\alpha}
\quad\;\hbox{and}\;\quad
\Le_s[\psi_n](s)
=\left[\smallfrac{s-\alpha}{s+\beta}\right]^n
\smallfrac{\sqrt{\alpha+\beta}}{s+\beta}.
$$
\end{lemma}

\begin{proof}
The functions $\phi_n$ and\/ $\psi_n$ were defined in Lemma 4.1 as follows.
$$
\phi_n(t)
=\sqrt{\alpha+\beta}\,e^{-\alpha t}L_n\big((\alpha+\beta)\kern1pt t\big)
\quad\;\hbox{and}\;\quad
\psi_n(t)
=\sqrt{\alpha+\beta}\ e^{-\beta t}\,L_n\big((\alpha+\beta)\kern1pt t\big).
$$
Recall: the Laplace transform of the Laguerre polynomial $L_n(t)$ is
$$
\Le_s[L_n](s)=\smallfrac{(s-1)^{n}}{s^{n+1}}.
$$
\vskip-4pt\noi
Thus (formally)
\vskip4pt\noi
\begin{eqnarray*}
\Le_s\big[\sqrt{\alpha+\beta}\,e^{-\alpha t}
L_n\big((\alpha+\beta)\kern1pt t\big)\big]
\kern-6pt&=&\kern-6pt
\sqrt{\alpha+\beta}
\;\smallfrac{(s+\alpha-[\alpha+\beta])^{n}}{(s+\alpha)^{n+1}}              \\
\kern-6pt&=&\kern-6pt
\sqrt{\alpha+\beta}
\;\smallfrac{(s-\beta)^n}{(s+\alpha)^{n+1}}
=\left[\smallfrac{s-\beta}{s+\alpha}\right]^n
\smallfrac{\sqrt{\alpha+\beta}}{s+\alpha},
\end{eqnarray*}
\vskip-4pt\noi
and hence
$$
\Le_s\big[\sqrt{\alpha+\beta}\,e^{-\beta t}
L_n\big((\alpha+\beta)\kern1pt t\big)\big]
=\left[\smallfrac{s-\alpha}{s+\beta}\right]^n
\smallfrac{\sqrt{\alpha+\beta}}{s+\beta}
$$
by interchanging $\alpha$ and $\beta$.
\end{proof}
\vskip-2pt

\vskip6pt
The shift operator $S$ corresponds to the operator {\it multiplication}\/ by
the function $H(\cdot)$ in the Hardy space $H^2$ which is given by
$$
H(s)=\smallfrac{s-\frac{1}{2}}{s+\frac{1}{2}}
$$
for ${\Real(s)>-\frac{1}{2}}.$ The Laplace
transforms of $\phi_n$ and $\psi_n$ in Lemma 4.2 imply the existence of
functions $H_{\alpha\beta}$ and $H_{\beta\alpha}$ in $H^2$ given by
$$
H_{\alpha\beta}(s)=\smallfrac{s-\alpha}{s+\beta}
\quad\;\hbox{and}\;\quad
H_{\beta\alpha}(s)=\smallfrac{s-\beta}{s+\alpha}.
$$
Consequently we consider the $(\alpha\beta)$-Laguerre operator
$S_{\alpha\beta}$ on $L^2[0,\infty)$ defined by ${S_{\alpha\beta}f=g}$
where, for every ${t\ge0}$
$$
(S_{\alpha\beta}f)(t)=g(t)
\quad\;\hbox{with}\;\quad
g(t)=f(t)-(\alpha+\beta)\int_0^t\!e^{-\beta(t-\tau)}f(\tau)\,d\tau,
$$
generating for each ${n\ge0}$ the function $\psi_n$ given by
$$
\psi_n(t)
=\sqrt{\alpha+\beta}\,e^{-\beta t}L_n\big((\alpha+\beta)\kern1pt t\big).
$$
Interchanging $\alpha$ and $\beta$ we have the $S_{\beta\alpha}$-Laguerre
operator $S_{\beta\alpha}$ on $L^2[0,\infty)$ defined by
${S_{\beta\alpha}f=g}$ where, for every ${t\ge0}$
$$
(S_{\beta\alpha}f)(t)=g(t)
\quad\;\hbox{with}\;\quad
g(t)=f(t)-(\alpha+\beta)\int_0^t\!e^{-\alpha(t-\tau)}f(\tau)\,d\tau,
$$
generating the function $\phi_n$ given by
$$
\phi_n(t)
=\sqrt{\alpha+\beta}\,e^{-\alpha t}L_n\big((\alpha+\beta)\kern1pt t\big).
$$
Observe that $S_{\alpha \beta}$ is associated with the $\alpha$-Laguerre
shift $S_{\alpha}$ while $S_{\beta\alpha}$ is associated with the
$\beta$-Laguerre shift $S_{\beta}$.

\vskip6pt
\begin{theorem}
The\/ $(\alpha\beta)$ and\/ $(\beta\alpha)$-Laguerre operators\/
$S_{\alpha\beta}$ and\/ $S_{\beta\alpha}$ consist of a biisometric pair on\/
$L^2[0,\infty)$ having the following properties.
$$
\N(S_{\alpha\beta}^*)=\span\{e^{-\alpha(\cdot)}\}=\N(S_\alpha^*),
$$
$$
\N(S_{\beta\alpha}^*)=\span\{e^{-\beta(\cdot)}\}=\N(S_\beta^*),
$$
$$
S_{\alpha\beta}^*S_{\beta\alpha}=I=S_{\beta\alpha}^*S_{\alpha\beta},
$$
$$
S_{\alpha\beta}S_{\beta\alpha}=S_\alpha S_\beta
=S_{\beta\alpha}S_{\alpha\beta},
$$
\vskip4pt\noi
and so\/ $S_{\alpha\beta}S_{\beta\alpha}$ is a shift of multiplicity\/ $2$.
\end{theorem}

\begin{proof}
For ${t\ge0}$
$$
S_{\alpha\beta}^*f=g
\quad\;\hbox{with}\;\quad
g(t)=f(t)-(\alpha+\beta)\int_t^\infty\!\!e^{-\beta(\sigma-t)}f(\tau)\,d\tau.
$$
Therefore $f\in\N(S_{\alpha\beta}^*)$ if and only if
${S^*_{\alpha\beta}f=0}$, which implies, for every ${t\ge0}$,
$$
f(t)=(\alpha+\beta)\int_t^\infty\!\!e^{-\beta(\sigma-t)}f(\tau)\,d\tau.
$$
Differentiating both sides we get, for ${t\ge0}$,
$$
f(t)=\beta f(t)-(\alpha+\beta)f(t)=-\alpha f(t).
$$
Solving for $f$ we get
\vskip-4pt\noi
$$
f(t)=K e^{-\alpha t}
$$
for ${t\ge0}$ and some constant $K.$ The same argument leads to
$$
S_{\alpha}^*e^{-\alpha t}=0
$$
for every ${t\ge0}.$ This proves the first property$.$ Interchanging $\alpha$
and $\beta$ we get the second one$.$ The next two properties are derived by
simple calculations$.$ Finally, since $S_{\alpha \beta}\,S_{\beta \alpha}$ is
the convolution of two commutable shifts of multiplicities $1$, viz.,
$\;S_{\alpha}S_{\beta}$, it is therefore a shift of multiplicity 2.
\end{proof}
\vskip-2pt

\vskip6pt
The next result follows from Corollary 3.1, Lemmas 4.1, 4.2, and Theorem 4.1.

\vskip6pt
\begin{corollary}
For each\/ ${n\kern-1pt\ge\kern-1pt0}$ consider the functions\/
${\phi_n,\psi_n\!\in\kern-1ptL^2[0,\infty)}$ as follows.
$$
\phi_n(t)
=\sqrt{\alpha +\beta}\;e^{-\beta t}L_n\big((\alpha+\beta)\kern1pt t\big)
=[S_{\alpha\beta}]^n\sqrt{\alpha+\beta}\,e^{-\beta t},
$$
$$
\psi_n(t)
=\sqrt{\alpha+\beta}\,e^{-\alpha t}L_n\big((\alpha+\beta)\kern1pt t\big)
=[S_{\beta\alpha}]^n\sqrt{\alpha+\beta}\,e^{-\alpha t}.
$$
The sequences\/ $\{\phi_n\}$ and\/ $\{\psi_n\}$ are biorthogonal and both
span\/ $L^2[0,\infty).$ Moreover, the biisometric operators\/
$S_{\alpha\beta}$ and\/ $S_{\beta\alpha}$ shift the biorthogonal sequences\/
$\{\phi_n\}$ and\/ $\{\psi_n\}.$ That is, for each\/ ${n\ge0}$
$$
S_{\alpha\beta}\phi_n=\phi_{n+1},
$$
and for every\/ ${f\in L^2[0,\infty)}$ with expansion in\/ $\{\phi_n\}$,
$$
f={\sum}_{k=0}^\infty\<f\,;\psi_{k}\>\phi_k
\quad\;\hbox{and so}\;\quad
S_{\alpha\beta}f={\sum}_{k=0}^\infty\<f\,;\psi_k\>\phi_{k+1}.
$$
Similarly, for each\/ ${n\ge0}$
$$
S_{\beta\alpha}\psi_n=\psi_{n+1},
$$
and for every\/ ${f\in L^2[0,\infty)}$ with expansion in\/ $\{\psi_n\}$,
$$
f={\sum}_{k=0}^\infty\<f\,;\phi_{k}\>\psi_k
\quad\;\hbox{and so}\;\quad
S_{\beta\alpha}f={\sum}_{k=0}^\infty\<f\,;\phi_k\>\psi_{k+1}.
$$
\end{corollary}

\begin{proof}
Apply Corollary 3.1, Lemmas 4.1, 4.2, and Theorem 4.1.
\end{proof}

\vskip6pt
\begin{remark}
(a) The Gram--Schmidt orthonormalization procedure can be extended to
biorthonormalization in Hilbert space$.$ Indeed, take a pair of sequences
$\{f_n\}$ and $\{g_n\}$, and construct the sequences $\{\phi_n\}$ and
$\{\psi_n\}$ so that ${\<\phi_m\,;\psi_n\>}=\delta_{m,n}$ as follows$.$
To begin with set
$$
\phi_0=\smallfrac{f_0}{\<f_0\,;g_0\>^\frac{1}{2}}
\quad\;\hbox{and}\;\quad
\psi_0=\smallfrac{g_0}{\<f_0\,;g_0\>^\frac{1}{2}},
$$
so that ${\<\phi_0\,;\psi_0\>=1}.$ Next set
$$
\phi_1=\smallfrac{p_1}{\<p_1\,;q_1\>^\frac{1}{2}}
\quad\;\hbox{and}\;\quad
\psi_1:=\frac{q_1}{\<p_1\,;q_1\>^\frac{1}{2}},
$$
\vskip-4pt\noi
where
$$
p_1=f_1-\<f_1\,;\psi_0\>\phi_0
\quad\;\hbox{and}\;\quad
q_1=g_1-\<g_1\,;\phi_0\>\psi_0.
$$
\vskip4pt\noi
It is plain that $\<\phi_1\,;\psi_1\>=1$, $\;p_1\perp\psi_0$, and
$q_1\perp\phi_0.$ Then
$$
r_1\phi_1= p_1=f_1-\<f_1\,;\psi_0\>\phi_0
\quad\;\hbox{where}\;\quad
r_{1}=\<p_1,q_1\>^\frac{1}{2}.
$$
Thus $\phi_1\perp\psi_0.$ Similarly,
$$
r_1\psi_1=q_1=g_1-\<g_1\,;\phi_0\>\psi_0,
$$
and so ${\psi_1\perp\phi_0}.$ In general, for ${n\ge1}$,
$$
r_n\phi_n=p_n=
f_n-{\sum}_{k=0}^{n-1}\<f_n\,;\psi_k\>\phi_k\perp\psi_0,\dots,\psi_{n-1}
\quad\;\hbox{where}\;\quad
r_n=\<p_n\,;q_n\>^\frac{1}{2}.
$$
Similarly,
\vskip-4pt\noi
$$
r_n\psi_n=q_n=
g_n-{\sum}_{k=0}^{n-1}\<g_n\,;\phi_k\>\psi_k\perp\phi_0,\dots,\phi_{n-1}.
$$
\vskip4pt\noi
If $f_n\perp h$, then $\phi_n\perp h.$ Therefore, if $\{f_n\}$ is total
(i.e., complete), then so is $\{\phi_n\}.$ Similarly, if $\{g_n\}$ is total
(i.e., complete), then so is $\{\psi_n\}$.

\vskip9pt\noi
(b) It is also worth noticing on the following points.
\vskip6pt\noi
\begin{description}
\item{$\kern-7pt$(i)}
If $f_n=g_n$, then the Gram--Schmidt biorthonormalization becomes the usual
Gram-Schmidt orthonormalization.
\vskip6pt\noi
\item{$\kern-9pt$(ii)}
If $f_n(t)=e^{-\alpha t}\big(({\alpha+\beta})\kern1pt t\big){\phantom|\!\!}^n$
and $g_n(t)=e^{-\beta t}\big(({\alpha+\beta})\kern1pt t\big){\phantom|\!\!}^n$,
then we get $\phi_n(t)=e^{\alpha t}L_n({\alpha+\beta})\kern1pt t$ and
$\psi_n(t)=e^{\beta t} L_n({\alpha+\beta})\kern1pt t.$ Moreover $\{f_n\}$ and
$\{g_n\}$ are total (i.e., complete), and $\phi_n$ and $\psi_n$ admit the
biisometric description as well$.$ Also $\{\phi_n\}$ and $\{\psi_n\}$ are in
this case independently total (i.e., complete).
\end{description}
\end{remark}

\section{Conclusion and Remarks}

We have seen in Section 4 that the $\alpha$-Laguerre shift $S_{\alpha}$
satisfied for each ${n\ge0}$ the following properly.
$$
S_{\alpha}\big(\sqrt{2\alpha}\,e^{-\alpha t}\,L_n(2\alpha t)\big)
=\sqrt{2\alpha}\,e^{-\alpha t}L_{n+1}(2\alpha t).
$$
Moreover,
\vskip-4pt\noi
$$
[S_{\alpha}^*]\,e^{-\alpha t}=0.
$$
\vskip4pt\noi
The same type of results can be obtained for $S_{\alpha\beta}.$ Indeed,
in the space $H^2$,
$$
[H_{\alpha\beta}]^n\Le_s\big[e^{-\alpha t}\big]
=\left[\smallfrac{s-\alpha}{s+\beta}\right]^n\smallfrac{1}{s+\alpha},
$$
\vskip-4pt\noi
and hence
$$
[H_{\alpha\beta}]^{n+1}\smallfrac{1}{s+\alpha}
=\left[\smallfrac{s-\alpha}{s+\alpha}\right]
\left[\smallfrac{s-\alpha}{s+\beta}\right]^n
\smallfrac{1}{s+\beta}
=H_{\alpha}[H_{\alpha\beta}]^n\smallfrac{1}{s+\beta},
$$
for each ${n\ge0}$, where $H_{\alpha}=\frac{s-\alpha}{s+\alpha}.$
Therefore for each\/ ${n\ge0}$
$$
[S_{\alpha\beta}]^n\,e^{-\alpha t}
=S_{\alpha}\big[e^{-\beta t}L_n\big((\alpha+\beta\big)\kern1pt t)\big].
$$
Interchanging $\alpha$ and $\beta$ we get
$$
[S_{\beta\alpha}]^n\,e^{-\beta t}
=S_\beta\big[e^{-\alpha t}L_n\big((\alpha+\beta)\kern1pt t\big)\big].
$$
These functions, however, are neither orthogonal nor biorthogonal$.$ The
$\alpha$-Laguerre functions and $\alpha$-Laguerre shift $S_{\alpha}$ have
been widely applied in systems analysis and identification, see for instance
\cite{MP} and the references therein$.$ Applications of biorthogonal Laguerre
functions and biisometric Laguerre shifts will be reported elsewhere.

\vskip6pt
Finally, the biisometric operators $S_{\alpha\beta}$ and $S_{\beta\alpha}$
discussed above can be regarded as ``Laguerre-like'' shifts of
multiplicity 1$.$ A class of Laguerre shifts and Laguerre shift semigroups
of finite multiplicities have been developed in \cite{LN}$.$ We expect that
one can also construct ``Laguerre-like" shifts of finite multiplicities.

\vskip4pt\noi
\section*{Acknowledgment}

We thank Ole Christensen for corrections in Theorem 3.1 and Corollary 3.1.

\bibliographystyle{amsplain}

\end{document}